\documentclass[12pt]{extarticle} 
\usepackage{graphicx} 
\usepackage{amsmath}

\usepackage[top=1in, bottom=1in, left=2in, right=2in, inner=2cm,outer=2cm]{geometry}

\usepackage{amsfonts}

\usepackage{amsmath}
\usepackage{amsthm}
\usepackage{amscd}
\usepackage{amssymb}

\usepackage{authblk}           

\usepackage{mathtools,tikz-cd}
\usetikzlibrary{cd}
\usepackage[all]{xy}
\usetikzlibrary{arrows,shapes}
\usepackage{cancel}

\usepackage{xcolor} 
\definecolor{EqColor}{HTML}{5E60CE}   
\definecolor{CiteColor}{HTML}{0A7AA1} 

\let\oldeqref\eqref
\renewcommand{\eqref}[1]{\textcolor{EqColor}{\oldeqref{#1}}}

\let\oldref\ref
\renewcommand{\ref}[1]{\textcolor{EqColor}{\oldref{#1}}}

\makeatletter
\let\orig@cite\cite
\renewcommand{\cite}{\@ifnextchar[{\my@cite@opt}{\my@cite@noopt}}
\newcommand{\my@cite@opt}[2][]{\textcolor{CiteColor}{\orig@cite[#1]{#2}}}
\newcommand{\my@cite@noopt}[1]{\textcolor{CiteColor}{\orig@cite{#1}}}
\makeatother

\newtheorem{theorem}{Theorem}[section]
\newtheorem{proposition}[theorem]{Proposition}
\newtheorem{lemma}[theorem]{Lemma}
\newtheorem{definition}[theorem]{Definition}
\newtheorem{corollary}[theorem]{Corollary}
\newtheorem{example}[theorem]{Example}

\newtheorem{remark}[theorem]{Remark}

\title{A novel approach through spherical functions in the characterization of invariant functions
}

\author{Rocío Díaz Martín and Linda Saal
}

\date{}

\begin{document}

\maketitle

\renewcommand{\thefootnote}{}
\footnotetext{%
  \begin{minipage}{\textwidth}
  \end{minipage}
}
\renewcommand{\thefootnote}{\arabic{footnote}}

\begin{abstract}
    \noindent Given a compact subgroup $K$ of the orthogonal group acting on the Euclidean space $\mathbb{R}^n$, Gerald Schwarz proved that every smooth $K$-invariant function on $\mathbb{R}^n$ can be expressed as a smooth function of a generating set of $K$-invariant polynomials on $n$ variables.
    The goal of this work is to provide an alternative and more straightforward proof of this result, based on Gelfand theory, with a particular focus on spherical functions.
\end{abstract}

\section{Introduction}

\subsection*{Schwarz's Theorem}
Given any closed subgroup of rotations and reflections acting naturally on the Euclidean domain $\mathbb{R}^n$, if one considers all the smooth functions  invariant under its action, it holds that they can be characterized as smooth functions on the generators of the algebra of invariant polynomials on $\mathbb{R}^n$.
Precisely, let $\textcolor{blue}{K}$ be a compact group acting orthogonally on $\mathbb R^n$.
From the classical theorem of Hilbert
(see \cite{Weyl}), it is known that the algebra of $K$-invariant polynomials on $\mathbb{R}^n$, denoted as $\textcolor{blue}{\mathcal{P}(\mathbb R^n)^K}$, is finitely generated. In 1975, Gerald Schwarz proved that all smooth $K$-invariant functions on $\mathbb{R}^n$,  $\textcolor{blue}{C^{\infty}\left(\mathbb{R}^n\right)^K}$, are characterized as smooth functions on an arbitrary system $\textcolor{blue}{\left\{\rho_1, \ldots, \rho_\ell\right\}}$ of  generators: 

\begin{theorem}\label{thm: schwarz} \cite[Thm 1]{Schwarz}
    Let $K$ be a compact group acting orthogonally on $\mathbb{R}^n$, and let $\left\{\rho_1, \ldots, \rho_\ell\right\}$ be a set of generators of the algebra $\mathcal{P}(\mathbb{R}^n)^K$. 
    Then, for every infinitely differentiable $K$-invariant function $f \in C^{\infty}\left(\mathbb{R}^n\right)^K$, there exists $h \in C^{\infty}\left(\mathbb{R}^\ell\right)$ such that $f(x) = h\left(\rho_1(x), \ldots, \rho_\ell(x)\right)$.
\end{theorem}
This result is very strong because it relates the differentiable structures of two spaces that in principle were only homeomorphic. Indeed, let
$\rho:\mathbb{R}^n\to\mathbb{R}^\ell$ given by
$$\textcolor{blue}{\rho(x)}:=(\rho_1(x),\dots,\rho_\ell(x)),$$ for every $x\in\mathbb{R}^n$. This map induces a homeomorphism between the space of orbits $\mathbb{R}^n/K$ and $\rho(\mathbb{R}^n)$:
$$\xymatrix{
\mathbb{R}^n \ar[rd]^{\rho} \ar[d]_{\text{quotient projection}}\\
\mathbb{R}^n/K \ar@{-->}[r] & \rho(\mathbb{R}^n) . }$$
As explained in \cite{Schwarz}, $\mathbb{R}^n/K$ can be given a smooth
structure by setting that a function on the quotient space $\mathbb{R}^n/K$ is smooth if when lifted to a $K$-invariant function on $\mathbb{R}^n$ it is smooth in the classical sense. At the same time,  the image set $\rho(\mathbb{R}^n)$, viewed as a closed subset of  $\mathbb{R}^\ell$, has a smooth structure by defining  that a function on $\rho(\mathbb{R}^n)\subset \mathbb R^\ell$ is smooth if it is the restriction to $\rho(\mathbb{R}^n)$ of a smooth function on $\mathbb{R}^\ell$. Thus, Theorem \ref{thm: schwarz} states that $\rho$ induces a homeomorphism of 
$\mathbb{R}^n/K$ and $\rho(\mathbb{R}^n)$ together
with their smooth structures. That is, by using \textit{pull-back} notation, Theorem \ref{thm: schwarz} states that the map $\rho^*$ below is onto 
\begin{gather*}
\rho^*:C^\infty(\mathbb R^\ell)\to C^\infty (\mathbb R^n)^K\\
\textcolor{blue}{ (\rho^*h)(x)}:=h(\rho_1(x),\dots,\rho_\ell(x)).
\end{gather*}

Theorem \ref{thm: schwarz} was first conjectured and  shown for some particular cases. For example, by using Taylor expansions and analytic extensions,  H. Whitney showed in 1943 that even smooth functions $f$ on $\mathbb R$ are of form $f(x)=h(x^2)$ for $h:\mathbb R\to \mathbb R$ smooth  \cite{Whitney}, that is, Theorem \ref{thm: schwarz} for $n=1$ and $K=\{\pm 1\}$ (see also \cite{fefferman} and \cite{Whitney_extension}). G. Glaeser, in 1963, extended 
Theorem \ref{thm: schwarz} for the case of the symmetric group $\mathbf{S}_n$ acting on $\mathbb R^n$ \cite{Glaeser}. 
Before the paper \cite{Schwarz} appeared with the general proof, Theorem \ref{thm: schwarz} was deduced for the case of finite groups \cite{Bierstone}.
\\

Finally, we observe that the image set $\rho^*(C^\infty(\mathbb R^\ell))$ is dense in $C^\infty(\mathbb R^n)^K$, as a consequence of the following basic facts:
\begin{enumerate}
    \item The space of invariant polynomials $\mathcal{P}(\mathbb R^n)^K$ is dense in $C^\infty(\mathbb R^n)^K$.\footnote{Fact 1 follows, on the one hand, from the density of polynomials in $C^\infty(\mathbb{R}^n)$ with respect to the Whitney topology (i.e., uniform convergence of the function and its derivatives on compact sets).
On the other hand, because the projection
$
C^\infty(\mathbb R^n)\ni f \longmapsto \Big(x \mapsto \int_K f(k\cdot x)\, dk\Big)\in C^\infty(\mathbb R^n)^K
$
is linear, continuous, and onto, and hence open by the Open Mapping Theorem. 
}
    \item At the polynomial level, it holds the identity
    $
    \mathcal{P}(\mathbb R^n)^K = \rho^*(\mathcal{P}(\mathbb R^\ell)).
    $\footnote{For the 2nd fact, the inclusion $\rho^*(\mathcal{P}(\mathbb R^\ell)) \subseteq \mathcal{P}(\mathbb R^n)^K$ is immediate since the polynomials $\rho_1,\dots,\rho_\ell$ are $K$-invariant. 
Conversely, given $p \in \mathcal{P}(\mathbb R^n)^K$, we can write it in terms of the generators $\rho_1,\dots,\rho_\ell$ as a finite sum $
p(x)=\sum a_{j,k}(\rho_j(x))^k$,
for some coefficients $a_{j,k}$, and defining $h(y_1,\dots,y_\ell) = \sum a_{j,k} \, (y_j)^k$ in $\mathbb R^\ell$, we obtain $
p(x)=h(\rho_1(x),\dots,\rho_\ell(x))$, which implies $\rho^*(\mathcal{P}(\mathbb R^\ell)) \supseteq \mathcal{P}(\mathbb R^n)^K$.}
\end{enumerate}
This suggests that a possible strategy to prove Theorem \ref{thm: schwarz} would be to establish the remaining step, namely that $\rho^*(C^\infty(\mathbb R^\ell))$ is closed in $C^\infty(\mathbb R^n)^K$. 
For instance, if the generators $\rho_1,\dots,\rho_\ell$ are algebraically independent, then $n \geq \ell$ (roughly speaking, one cannot have more generators than variables), and by a result of G. Glaeser \cite{Glaeser} it follows that $\rho^*(C^\infty(\mathbb R^\ell))$ is closed in $C^\infty(\mathbb R^n)^K$. 
In this work, however, we will adopt a different approach.

\subsection*{Our Contributions}

Our main objective is to provide an alternative --possibly more elementary-- proof of Theorem \ref{thm: schwarz} employing \textit{Gelfand theory}\footnote{The original proof by Gerald Schwarz relies on \textit{Grothendieck's theory of topological tensor products and nuclear spaces} \cite{gro,sch,treves}.}. We will leverage the fact that, for any compact subgroup $K$ of the orthogonal group $ \mathrm{O}(n)$ one has that
$\textcolor{blue}{(K\ltimes\mathbb{R}^n, K)}$ (in short, $\textcolor{blue}{(K,\mathbb R^n)}$) forms a \textit{Gelfand pair}, 
where $K \ltimes \mathbb{R}^n$ denotes the semidirect product of $K$ and $\mathbb{R}^n$ (known as the $n$-dimensional connected Euclidean motion group when $K=\mathrm{SO}(n)$, i.e., the group of isometries of $\mathbb{R}^n$). 
Specifically, by employing \textit{spherical functions}, we aim to explicitly determine, for a given $K$-invariant function $f$, a corresponding function $h$ in Theorem \ref{thm: schwarz} satisfying $h(\rho(x))=f(x)$. 

Our main contribution is the proof of a version of Theorem \ref{thm: schwarz} for integrable $K$-invariant functions $f:\mathbb R^n \to\mathbb C$ whose Fourier transform has compact support  (see Theorem \ref{thm: schwarz2} below). Such functions $f$ are, in particular, not only smooth but in fact real analytic (by the Paley–Wiener theorem). As a counterpart, we establish stronger regularity for the associated function $h:\mathbb R^\ell\to \mathbb C$ than in Theorem \ref{thm: schwarz}, namely, we show that $h$ may be chosen to be real analytic as well.

In addition, the problem addressed by Gerald Schwarz in \cite{Schwarz} can be reformulated, within the framework of Gelfand theory, as an \textit{extension problem} for the Gelfand transform  (see Corollary \ref{eq: nuestro thm} and the description below). The precise definitions and underlying concepts will be introduced and motivated in Section \ref{sec: prelim}. The purpose of the upcoming discussion is simply to provide preliminary connections between this theory and Theorem \ref{thm: schwarz}.

\paragraph{Schwarz's Theorem and Gelfand Theory:} Given $f:\mathbb{R}^n\to \mathbb C$ an integrable $K$-invariant function, we write $f\in \textcolor{blue}{L^1(\mathbb R^n)^K}$. That is, if $\textcolor{blue}{k\cdot x}$ denotes the natural action of $k\in K$ on $x\in \mathbb R^n$, then $f(k\cdot x)=f(x)$ for all $k\in K$, $x\in \mathbb R^n$. We recall that such natural action $k\cdot x$ can be represented as a matrix-vector multiplication as $K$ is a subgroup of the group of orthogonal $n\times n$ matrices. 
The (classical) Fourier transform of such function $f$ is well-defined  as  
\begin{equation}\label{eq: class FT}
 \textcolor{blue}{ \widehat{f}(\xi)}:=\int_{\mathbb{R}^n} f(x) \,  e^{-i\langle x,\xi\rangle} \, dx.  
\end{equation}
Since the convolution algebra $L^1(\mathbb R^n)^K$ is always commutative (regardless of the choice of $K$), we have a so-called \textit{Gelfand pair} $(K,\mathbb R^n)$. Then, the Gelfand theory provides a well-defined framework for the notion of the \textit{spectrum} of the algebra $L^1(\mathbb R^n)^K$, along with the so-called 
\textit{Gelfand transform}. Indeed, the Gelfand transform of $f$ is defined as 
\begin{equation}\label{eq: Gelf t}
  \textcolor{blue}{\mathcal{F}({f})(\varphi_\xi)}:=\int_{\mathbb{R}^n} f(x) \, \varphi_\xi(-x) \, dx, \qquad \text{ where } \quad  \textcolor{blue}{\varphi_\xi(x)}:= \int_{K} e^{i\langle x,k\cdot\xi\rangle} \, dk,  
\end{equation}
where $dk$ denotes the normalized Haar measure on $K$. The set of (bounded) \textit{spherical functions} $\{\varphi_\xi\}$ determines the spectrum \textcolor{blue}{$\Sigma$}  of the algebra $L^1(\mathbb R^n)^K$. 
As $f$ is $K$-invariant, its Fourier transform $\widehat{f}$ is also a $K$-invariant function, and coincides with its Gelfand transform, that is, 
\begin{equation}\label{eq: geltansd and fourier}
    \mathcal{F}(f)(\varphi_\xi)=\widehat{f}(\xi)
\end{equation}
(think of integrating \eqref{eq: class FT} over $K$ to obtain \eqref{eq: Gelf t}).
Moreover, from the seminal paper \cite{ferrari_ruffino}, the domain $\Sigma$ of the Gelfand transform can be identified with a closed subset \textcolor{blue}{$\Lambda$} of \( \mathbb{C}^\ell \). This result holds in great generality. In particular, we will show that \( \Lambda \) can be identified with the image of \( \mathcal{\rho} \) in \( \mathbb{R}^\ell \), allowing us to write
\begin{equation}\label{eq: f transf en rho}
\mathcal{F}(f)(\varphi_\xi) = \mathcal{F}(f)(\rho(\xi)).
\end{equation}
Using this notation within the Harmonic Analysis framework, in this work we will prove the following proposition and the subsequent version of Theorem \ref{thm: schwarz}:

\begin{proposition}\label{prop: aux sph f} Let $(K,\mathbb R^n)$ be a Gelfand pair,  let $\varphi _{\xi }$ be an associated bounded  spherical function, and let $\left\{\rho_1, \ldots, \rho_\ell\right\}$ be a set of generators of the algebra $\mathcal{P}(\mathbb{R}^n)^K$. Then, there exists a  real analytic function  $h_\xi$ on $\mathbb R^\ell$ such that   $\varphi_\xi(x)=h_\xi(\rho_1(x), \ldots, \rho_\ell(x))$ for every $x\in \mathbb R^n$. 
\end{proposition}

\begin{theorem}\label{thm: schwarz2}
    Let $K$ be a compact group acting orthogonally on $\mathbb{R}^n$, and let $\left\{\rho_1, \ldots, \rho_\ell\right\}$ be a set of generators of the algebra $\mathcal{P}(\mathbb{R}^n)^K$. 
    Then, for every $K$-invariant function $f\in L^1(\mathbb R^n )$  with $\widehat{f}$ of compact support, there exists a real analytic 
    function $h:\mathbb R^\ell\to \mathbb C $ such that $f(x) = h\left(\rho_1(x), \ldots, \rho_\ell(x)\right)$. 
\end{theorem}

Closed-formulas for choosing the functions $h_\xi$ and $h$ will be provided in Section \ref{sec: sph} (Definition \ref{def: h_xi}) and Section \ref{sec: main} (Definition \ref{def: h}), respectively.

As a consequence, Gerald Schwarz's result may be interpreted as showing that for a smooth integrable $K$-invariant function $g$ with compact support, the Gelfand transform  $\mathcal{F}(g):\Lambda\to \mathbb C$ admits a smooth extension $h:\mathbb R^\ell \to \mathbb C$, that is, $\mathcal{F}(g)(\rho(\xi))=h(\rho(\xi))$.  
Indeed, for such a function $g$, consider the function $f=\widehat{g}$. Applying  Theorem \ref{thm: schwarz2}, we obtain that for this $f:\mathbb R^n\to \mathbb C$, there exists a regular function $h:\mathbb R^\ell\to \mathbb C$ such that
$$h(\rho(x))=f(x)=\widehat{g}(x)=\mathcal{F}(g)(\rho(x)),$$
where we have used \eqref{eq: geltansd and fourier} and \eqref{eq: f transf en rho}.
Hence, another contribution of this work can be  understood as addressing the extension problem for the Gelfand transform in the setting of ``abelian pairs'', i.e.,  Gelfand pairs of the form $(K,\mathbb R^n)$. Specifically, the main result of this paper, Theorem \ref{thm: schwarz2}, can be also read in terms of the following corollary:

\begin{corollary}\label{eq: nuestro thm}
    Given $g\in C^\infty(\mathbb R^n)^K$ with compact support, consider $h:\mathbb R^\ell \to \mathbb C$ of the form 
    \[
    h( t) =\int_{\mathbb{R}^n}  g(- \xi )\,  h_{\xi }(
    t) \, d\xi, \quad \text{ for }  h_{\xi } \text{ given in Proposition \ref{prop: aux sph f}.}
    \]
    Then, it holds that $h\in C^\infty(\mathbb R^\ell)$ and  $\mathcal{F}(g)(\varphi_\xi)=h(\rho(\xi))$.
\end{corollary}

It is important to mention that in recent years, the study of the Gelfand transform for Gelfand pairs arising from semidirect products \( K \ltimes N \), where \textcolor{blue}{\( N \)} is a connected, simply connected  nilpotent Lie group and \( K \) acts on \( N \) by automorphisms, has received considerable attention. Indeed, several results concerning the extension of the Gelfand transform have already been established in the literature (see, for e.g.,  \cite{ABR, ABR2, F1, F2}).
Specifically, if $(K\ltimes N, K)$ (or simply, \( (K, N) \)) is a Gelfand pair,  it can be shown that the associated spectrum \( \Lambda \)  can be embedded in a real space \( \mathbb{R}^\ell \). 
Let \textcolor{blue}{ \( \mathcal{S}(\Lambda) \)} denote the space of functions \( f: \Lambda \to \mathbb{C} \) that admit a Schwartz-class extension to \( \mathbb{R}^\ell \) (that is, $\mathcal{S}(\Lambda)=\mathcal{S}(\mathbb R^\ell)/\sim$, endowed with the quotient topology, where $f_1\sim f_2$ if and only if $f_1(\lambda)=f_2(\lambda)$ for all $\lambda\in \Lambda$).
In this setting, the Gelfand transform becomes an isomorphism from the space of $K$-invariant Schwartz functions over $N$,  \textcolor{blue}{\( \mathcal{S}(N)^K \)},  onto \( \mathcal{S}(\Lambda) \). That is, for every \( g \in \mathcal{S}(N)^K \), its Gelfand transform \( \mathcal{F}(g) \) admits a Schwartz extension to \( \mathbb{R}^\ell \). Moreover, a control of \( \mathcal{F} \) can be established in terms of the seminorms defining the topologies of \( \mathcal{S}(N)^K \) and \( \mathcal{S}(\mathbb{R}^\ell) \).

For the particular abelian case \( N = \mathbb{R}^n \), 
in \cite[Section 6]{ABR2} the authors prove the extension results for the Gelfand transform between Schwartz spaces by first invoking Theorem \ref{thm: schwarz} by Gerald Schwarz in \cite{Schwarz} and relying on the work \cite{Mather}.
Instead, in this work our goal is to prove from scratch Theorem \ref{thm: schwarz} by appealing to techniques of the Gelfand theory. It is remarkable that the arguments employed in the aforementioned papers \cite{ABR, ABR2, F1, F2} are very deep. The idea of this article is to be as elemental as possible.

\subsection*{Organization of the Paper} In Section \ref{sec: prelim}, we will elaborate on all the terminology briefly introduced above. In our review of Gelfand’s theory, we introduce the notions of Gelfand pairs, spherical functions, and the associated Gelfand transform, with a focus on the Euclidean case. In Section \ref{sec: new proof}, we present our main contributions. We begin in Section \ref{sec: sph} by considering $f$ as a spherical function ($f=\varphi_\xi$) in the statement of Theorem \ref{thm: schwarz}, which leads to Definition \ref{def: h_xi} and the proof of Proposition \ref{prop: aux sph f}. Building on these key components, in Section \ref{sec: main} we provide the proof of Theorem \ref{thm: schwarz2}, that is, our version of Gerald Schwarz’s theorem from the perspective of Gelfand theory. Finally, in Section \ref{sec: ex} we illustrate our results with examples.

\section{Preliminaries}\label{sec: prelim}

Gelfand's theory is devoted to studying commutative Banach algebras and their \textit{spectrums}. 
One of the most important results from Gelfand's theory establishes that a commutative Banach algebra $\textcolor{blue}{\mathcal{A}}$ can be mapped through a continuous group homomorphism --the so-called \textit{Gelfand transform}-- into an algebra of continuous functions defined over the spectrum of $\mathcal{A}$ (see, for example, \cite{Dijk,Folland}). 

As an application to Harmonic Analysis of the Gelfand theory on commutative Banach algebras, given a locally compact Hausdorff topological group $\textcolor{blue}{G}$, natural functional spaces to consider are  the \textit{group algebra} $C_c(G)$ of complex-valued continuous functions on $G$ with compact support 
endowed with the convolution product \textcolor{blue}{$*$},  or its closure under the $L^1$-norm, that is, the space of integrable functions $L^1(G)$. To apply  Gelfand's theory, the convolution product must be commutative. As this happens if and only if the underlying group $G$ is abelian, there has been a great interest in determining \textit{subalgebras} of $L^1(G)$ that are commutative under the convolution product
and invariant under the action of a subgroup of $G$. This yields to the definition of \textit{Gelfand pairs} (see, e.g., \cite{BJR,Carcano,ferrari_ruffino,F1,F2,Kikuchi,Vinberg}):
Given $K$ a compact subgroup of $G$, we say that $\textcolor{blue}{(G,K)}$ is a \textit{Gelfand pair}, or that the homogeneous space $G/K$ is \textit{commutative}, if the convolution subalgebra $\textcolor{blue}{L^1(G)^K}$ of bi-$K$-invariant integrable functions (i.e., $f(k_1xk_2)=f(x)$ for all $k_1,k_2\in K$, $x\in G$) is commutative. 
When having $(G,K)$ a Gelfand pair, the Gelfand transform on the commutative Banach algebra $L^1(G)^K$ plays the role of the classical Fourier transform. Indeed, in the context of Fourier analysis on groups, the Gelfand transform is often called the \textit{spherical Fourier transform}.

To define such a transform we need to introduce the spectrum of the algebra $L^1(G)^K$, or the so-called set of spherical functions. 
We say that a function $\varphi\in C(G)^K$ is \textit{spherical} if the associated \textit{character} $\chi_\varphi$, defined by
\begin{equation*}
    \textcolor{blue}{\chi_\varphi(f)}:=\int_G f(x) \, \varphi(-x) \, dx \qquad (\forall f \in C_c(G)^K),
\end{equation*}
satisfies
\begin{equation*}
    \chi_\varphi(f{*}g)=\chi_\varphi(f)\chi_\varphi(g),
\end{equation*}
In this case, the spectrum $\Sigma$ of the algebra $L^1(G)^K$ can be identified with the set of all bounded spherical functions \cite{helgason}. Then, 
$\Sigma\subset L^\infty(G)^K$ and it is  endowed with the \textit{Gelfand topology}, that is, the weak*-topology (i.e., we say that $\phi_n\to \phi$ as $n\to \infty$ if and only if $\chi_{\phi_n}(f)\to \chi_\phi(f)$ as $n\to \infty$, for all $f\in L^1(G)^K$) \cite{ferrari_ruffino}. It can be proven that this topology coincides with the topology of the uniform convergence on compact sets. 
When $G$ and $K$ are Lie groups, the spherical functions can be characterized by a differential point of view. 
Finally, the spherical Fourier transform of a function $f\in L^1(G)^K$ 
is defined
as the Gelfand transform associated to the commutative algebra $\mathcal{A}=L^1(G)^K$: as the function $\mathcal F(f): \Sigma\to\mathbb C$ given by the formula 
\begin{equation}\label{eq: GT}
    \mathcal{F}(f)(\varphi)=\int_G f(x) \, \varphi(-x) \, dx,
\end{equation}
where $\textcolor{blue}{-x}$ denotes the inverse of $x$ in the group $G$ (a particular case was introduced in \eqref{eq: Gelf t}).

A relevant family of study is when $G$ is a semidirect product $G=K\ltimes N$, where $N$ is a connected and  simply connected nilpotent Lie group and $K$ acts by automorphisms on $N$ (see, for e.g., \cite{Wolf,BJR}). 
Here, $K \ltimes N$ denotes the semidirect product of $K$ and $N$, that is, the manifold $K \times N$ equipped with the group product 
\[
\textcolor{blue}{(k, x)(k', x')}: = (kk', x + k \cdot x'),
\]
where $k\cdot x$ denotes the action of $k\in K$ on $x\in \mathbb N$, and $+$ is the group operation  in $N$ (not necessarily abelian)
giving it the structure of a Lie group. 
In this case, the algebra of bi-$K$-invariant functions $L^1(K\ltimes N)^K$ can be identified with the algebra $L^1(N)^K$ of $K$-invariant functions
$$\textcolor{blue}{L^1(N)^K}:=\{f\in L^1(N): \, f(k\cdot x)=f(x) \quad \forall k\in K, x\in N \},$$
and for simplicity in the notation it is commonly used $(K,N)$ in place of the pair $(K\ltimes N,K)$. 
{To see this, first note that if \( f: K \ltimes N \to \mathbb{C} \) is invariant under the right action of \( K \), then
\[
f(k, x) = f\left((\mathrm{Id}, x)(k, 0)\right) = f(\mathrm{Id}, x),
\]
where \( \mathrm{Id} \) denotes the identity element of \( K \). Thus, \( f(k, x) \) can be identified with the function \( f(\mathrm{Id}, \cdot): N \to \mathbb{C} \).
Now, suppose that \( f \) is also invariant under the left action of \( K \). Then,
\[
f((k,0)(\mathrm{Id}, x)) = f(k, k \cdot x) = f(\mathrm{Id}, k \cdot x),
\]
where the last equality follows from the computation above. This implies that the function \( f(\mathrm{Id}, \cdot) \) is invariant under the natural action of \( K \) on \( N \).
}

In these cases, it is well known that the spherical functions are of positive type \cite{BJR}
and can be characterized 
as the eigenfunctions of the algebra $\textcolor{blue}{\mathbb{D}(N)^K}$ of all the differentiable operators on $N$ invariant by $N$-left translations and invariant under the action of $K$ 
\begin{equation}\label{eq: eigenfunc}
    D\varphi=\lambda\varphi \qquad \forall D\in \mathbb{D}(N)^K,
\end{equation}
normalized by taking the value $1$ at the group identity.
Such algebra $\mathbb{D}(N)^K$ is finitely-generated \cite{helgason2}, and we denote by $\textcolor{blue}{\{D_1,\dots,D_\ell\}}$ a system of generators. Moreover, it holds that the spherical functions are analytic, and they are completely determined by the eigenvalues 
with respect to such arbitrary set of generators (see \cite[page 400]{helgason}).
Then, to each spherical function $\varphi$ we can associate an $\ell$-tuple of eigenvalues $\textcolor{blue}{(\lambda_1(\varphi),\dots, \lambda_\ell(\varphi))}$ with respect to the differential operators $D_1,\dots, D_\ell$. Let
\begin{equation}\label{eq: lambda set}
    \textcolor{blue}{\Lambda}:=\left\{(\lambda_1(\varphi),\dots, \lambda_\ell(\varphi)):\, \varphi\in \Sigma\right\}
\end{equation}

\begin{theorem}\cite{ferrari_ruffino}\label{thm: ff}
     $\Lambda$ is a closed subset of $\mathbb{C}^\ell$ and the correspondence $$\Sigma\ni \varphi\mapsto (\lambda_1(\varphi),\dots, \lambda_\ell(\varphi))\in \Lambda$$ is a homeomorphism between $\Sigma$, with the Gelfand topology, and $\Lambda$, with the relative topology of $\mathbb C^\ell$.
\end{theorem}

In this work, we will consider the abelian or Euclidean case, that is, $N=\mathbb R^n$. 
The classical Fourier transform on the commutative convolution algebra $L^1(\mathbb{R}^n)$ (defined in \eqref{eq: class FT})
is a continuous injection into the space of continuous functions which vanish at infinity $C_0(\mathbb{R}^n)$  (\textit{Riemann-Lebesgue Lemma}), which carries convolutions to point-wise products.
Indeed, the classical Fourier transform can be identified with the Gelfand transform associated to the  Gelfand pair $(K,\mathbb R^n)$, where the $K$ is the trivial group $K=\{\mathrm{Id}\}$.

As $L^1(\mathbb{R}^n)$ is commutative with the convolution product,  the subalgebra
\textcolor{blue}{$L^1(\mathbb{R}^n)^K$}
is always commutative. Thus, $(K,\mathbb{R}^n)$ is a Gelfand pair for any compact subgroup $K$ of the orthogonal group $\mathrm{O}(n)$ acting naturally on $\mathbb R^n$.

Given $f\in L^1(\mathbb R^n)^K$,  its  Gelfand transform \eqref{eq: GT} takes the form  \eqref{eq: Gelf t} where the 
the {spherical functions}, denoted by $\varphi_\xi$,  can be  determined by the integral formula also given  in \eqref{eq: Gelf t}.
Notice that each function $\varphi_\xi$ in \eqref{eq: Gelf t} is infinitely differentiable, $K$-invariant, bounded by one, and of positive type. Moreover, $\varphi_\xi=\varphi_{\xi'}$ if and only if $\xi=k\cdot \xi'$ for some $k\in K$.
We refer the reader to \cite{Wolf} for a nice study of spherical functions on Euclidean spaces. As a consequence of the definition \eqref{eq: Gelf t}, we highlight a very simple fact, namely, their symmetry:
\begin{equation}\label{eq: sym}
    \varphi_\xi(x)=\varphi_x(\xi) \qquad (\forall x,\xi\in \mathbb R^n).
\end{equation}

Notice that, given $f\in L^1(\mathbb R^n)^K$, by using our definitions and the invariance of $f$ under $K\subseteq \mathrm{O}(n)$, we can relate the spherical Fourier transform of $f$ and the classical Fourier transform of $f$ recovering the equality \eqref{eq: geltansd and fourier}: 
\begin{align*}
 \mathcal{F}({f})\left( \varphi_\xi \right)&=\int_{\mathbb{R}^n} f(x) \varphi
_{\xi }\left( -x\right) dx=\int_{\mathbb R^n } f(x) \int_{K}e^{-i\left\langle
x,k\cdot\xi \right\rangle }dk \, dx\\
&=\int_{\mathbb R^n } \int_{K} f\left( k \cdot x\right) e^{-i\left\langle
x,\xi \right\rangle }dk \, dx=\int_{\mathbb R^n } \int_{K} f(x) e^{-i\left\langle
x,\xi \right\rangle }dx
=\widehat{f}\left(\xi \right) . 
\end{align*}
Hence, under the hypothesis of the 
\textit{inversion formula} for the classical Fourier transform, we can write
\begin{equation}\label{eq: inv formula with xi}
    f(x)=\int_{\mathbb{R}^n}\mathcal{F}(f)(\varphi_\xi) \, \varphi_\xi(x) \, d\xi, 
\end{equation}
where $d\xi$ denotes the Lebesgue measure on $\mathbb R^n$ multiplied by $\frac{1}{(2\pi)^n}$.
Moreover, there exists a \textit{Plancherel measure} $\mu$, that is, a positive Borel measure on $\Lambda$,  such that an \textit{inversion formula} of the following form
\begin{equation*}
    f(x)=\int_{\Lambda}\mathcal{F}(f)(\varphi_\xi) \, \varphi_\xi(x) \, d\mu(\varphi_\xi) 
\end{equation*}
holds under the usual integrability conditions, and where we identify $\Lambda$ with $\Sigma$ due to Theorem \ref{thm: ff}.

\section{New proof of Schwarz's Theorem}\label{sec: new proof}

As in the statement of Theorem \ref{thm: schwarz}, let $K$ be a compact group acting orthogonally on $\mathbb{R}^n$, and let $\left\{\rho_1, \ldots, \rho_\ell\right\}$ be a set of generators of the algebra $\mathcal{P}(\mathbb{R}^n)^K$ of $K$-invariant polynomials on $\mathbb{R}^n$.
Clearly, we can assume that the polynomials $\rho _{1},...,\rho _{\ell}$ are homogeneous. 

Let us denote by $D_j$ the differential operator corresponding
to the polynomial $\rho_{j}$ (sometimes denoted as $\partial_{\rho_j}$), which, in simple words,  is obtained from $\rho_j$ changing the variables $x_k$ by the differential operators ${\partial_{x_k}}$, that is,\footnote{This can be generalized to abstract Lie groups through the so-called \textit{symmetrization map}  (see \cite[Theorem 4.3]{helgason}). In our case, since $\mathbb{R}^n$ is an abelian group, such map is the identity, that is, $x_k\mapsto {\partial_{x_k}}$.}
\begin{equation*}
    \text{if } \quad \rho_j(x)=\sum_{I=(i_1,\dots,i_n)}c_{I} \, x_1^{i_1} \, \cdot...\cdot x_n^{i_n}, \qquad \text{ then } \quad  D_j=\sum_{I=(i_1,\dots,i_n)}c_{I} \, \partial_{x_1}^{i_1} \, \cdot...\cdot \partial_{x_n}^{i_n}.
\end{equation*}
We observe that $D_{j}$ is left-invariant,
that is, a differential operator with constant coefficients, and moreover, it is $K$-invariant (i.e., for every differentiable $K$-invariant function $f$,  $D_{j}f$ is also $K$-invariant). Indeed, $\{D_1,\dots, D_\ell\}$ forms a system of generators of the algebra $\mathbb{D}(\mathbb R^n)^K$ of all $K$-invariant differentiable operators with constant coefficients.

By the integral expression \eqref{eq: Gelf t} of spherical functions, we have 
\begin{equation}\label{eq: D varphi = lambda varphi}
    D_j\varphi _{\xi }\left( x\right) =i^{\mathrm{deg}(\rho _{j})}\int_{K}\rho _{j}\left(
k\cdot\xi \right) e^{i\left\langle x,k\cdot \xi \right\rangle }dk=i^{\mathrm{deg}(\rho_{j})}\rho
_{j}(\xi) \varphi _{\xi }(x) 
\end{equation}
since $\rho _{j}$ is $K$-invariant. Thus, $\varphi_\xi$ is an eigenfunction of each $D_j$ corresponding to eigenvalue
\begin{equation}\label{eq: eigval}
    \lambda _{j}(\xi) :=i^{\mathrm{ deg}(\rho_{j})}\rho _{j}(\xi) . 
\end{equation}
Note that for simplicity in the notation we are writing $\lambda _{j}(\xi)$ instead of $\lambda _{j}\left( \varphi_\xi \right)$.

\begin{remark}\label{lem: spect in Rl} As a particular case of Theorem \ref{thm: ff}, the spectrum $\Sigma$ of the algebra $L^1(\mathbb R^n)^K$ is in correspondence with the image of $\rho=(\rho_1,\dots,\rho_\ell)$. Indeed, it is enough to change the above system of generator  $\{D_j\}_{j=1}^\ell$ of the algebra $\mathbb{D}(\mathbb R^n)^K$ by scalar multiples 
    $$\textcolor{blue}{\widetilde{D_j}}:=(-i)^{\mathrm{ deg}(\rho_{j})}D_j \qquad \text{ for } j=1,\dots,\ell.$$ 
    Then, from \eqref{eq: D varphi = lambda varphi} we obtain that the eigenvalues of the new operators $\widetilde{D_j}$ corresponding to the eigenfunctions $\varphi_\xi$   are
    $\rho_{j}(\xi)$.
       This shows that $\Lambda$ can be identified with the image set $\rho(\mathbb R^n)$.
\end{remark}

In what follows, we will use the following notation.  Given 
$\ell$ and $n$ dimensional multi-indexes $J=(j_1,\dots,j_\ell)$ and $I=(i_1,\dots, i_n)$, we denote 
\begin{equation}\label{eq: deg}
 \textcolor{blue}{|J|_{\ell,\rho}}:=j_{1}\mathrm{deg}(\rho _{1})+...+j_{\ell}\mathrm{deg}(\rho _{\ell}), \qquad \textcolor{blue}{|I|_n}=i_1+\dots+i_n .
\end{equation} 
We will also use the notation
\begin{equation*}
  \textcolor{blue}{\varrho(x)}:=\Pi_{k=1}^\ell\rho_k(x),  \qquad \textcolor{blue}{\mathbf{x}}:=\Pi_{k=1}^n x_k, \qquad \textcolor{blue}{D}:=D_1\dots D_\ell, \qquad \textcolor{blue}{\partial}:=\partial_{x_1}\dots\partial_{x_n},
\end{equation*}
where the first two are polynomials and the last two are differential operators. So, $\varrho(x)^J:=\Pi_{k=1}^\ell\rho_k(x)^{j_k}$, $\mathbf x^I:=\Pi_{k=1}^n x_k^{i_k}$,  $D^J:=D_1^{j_1}\dots D_\ell^{j_\ell}$, and $\partial^I=\partial_{x_1}^{i_1}\dots\partial_{x_n}^{i_n}$.

\subsection{Taylor expansions of spherical functions}\label{sec: sph}

It is well known that every {bounded} spherical function is real analytic \cite{helgason}. Thus, we write 
\begin{equation}\label{eq: real analytic}
    \varphi _{\xi }(x) =\sum_{m=0}^\infty p_{m, \xi }\left(x\right)   
\end{equation}
where each $\textcolor{blue}{p_{m,\xi}}$ is a homogeneous polynomial of degree $m$ and where the convergence of the series is absolute and uniform  over compacts. Moreover, since the action of $K$ on $\mathbb R^n$ is linear, each $p_{m,\xi}$ is $K$-invariant. Thus, by using the generators $\{\rho_1,\dots, \rho_\ell\}$ of $\mathcal{P}(\mathbb R^n)^K$, one can express each polynomial $p_{m,\xi}$ as
\begin{equation}\label{eq: pol}
p_{m,\xi}(x) =\sum_{J_m: \, |J_m|_{\ell,\rho}=m}a_{J_m,\xi} \, \varrho (x) ^{J_m}=\sum_{I_m:\, |I_m|_n=m}b_{I_m,\xi} \, \mathbf x^{I_m}, 
\end{equation}
for appropriate coefficients $\textcolor{blue}{a_{J_m,\xi}}$ and $\textcolor{blue}{b_{I_m,\xi}}$,
where the first sum runs over multi-indexes $J_m=(j_1,\dots,j_\ell)$, and the second sum is the Taylor expansion running over multi-indexes $I_m=(i_1,\dots,i_n)$. Thus, let us write
\begin{equation}\label{eq: spherical function as a series}
     \varphi _{\xi }(x) =\sum_{J=(j_1,\dots,j_\ell)}a_{J,\xi} \, \varrho (x)^{J}=\sum_{I=(i_1,\dots,i_n)}b_{I,\xi} \,  \mathbf x^{I},
\end{equation}
Then, 
\begin{equation}\label{eq: tay 1}
    (\partial^I\varphi_\xi)(0)=I! \, b_{I,\xi}, \qquad \text{where } \textcolor{blue}{I!}:=i_1!\cdot ... \cdot i_n!
\end{equation}
As in \cite[expressions (2) and (3) in the proof of Prop. 2.2, page 400]{helgason}, for each multi-index $I=(i_1,\dots, i_n)$, we consider the differential operator 
\begin{equation*}
\textcolor{blue}{\partial_0^I}:=\int_K \left(Ad(k)\partial^I\right) \, dk,
\end{equation*}
where adjoint action 
$Ad(k)$ on the differential operator  
$\partial$ is defined by conjugation,
that is, for all sufficiently differentiable  $f$ on $\mathbb R^n$, 
\begin{equation*}
  \left( Ad(k)\partial^I f \right)(x) =\left( \partial^I g \right)(k\cdot x), \quad \text{ where } g(y):=f(k^{-1}\cdot y).
\end{equation*}
Moreover, for  all sufficiently differentiable $K$-invariant functions $f$  on $\mathbb R^n$ it holds that
\begin{equation*}
    \left(\partial^I f\right)(0)=\left(\partial_0^I f\right)(0). 
\end{equation*}
Therefore, $\partial_0^I\in \mathbb D(\mathbb R^n)^K$ and so it can be written as a polynomial $\textcolor{blue}{Q_I}$ on the generators $D_1,\dots, D_\ell$:
\begin{equation*}
    \partial_0^I=Q_I(D_1,\dots,D_\ell)    .
\end{equation*}
We note that the order of the differential operator $\partial^{I_m}$ is $|I_m|_n=m$, and so it is the order of differential operator $\partial_0^{I_m}$. 

As a consequence, using \eqref{eq: eigenfunc} and Remark \ref{lem: spect in Rl}, we have that:
\begin{equation}\label{eq: b_I as pol in lambda}
   b_{I,\xi} = \frac{1}{I!}(\partial^I\varphi_\xi)(0)=\frac{1}{I!}(\partial_0^I\varphi_\xi)(0)
   =\frac{1}{I!}Q_I(\rho_1(\xi),\dots,\rho_\ell(\xi)).
\end{equation}
Besides, from \eqref{eq: pol}, for each $\ell$-dimensional multi-index $J_m$ with $|J_m|_{\ell,\rho}=m$, the  coefficient $a_{J_m,\xi}$ depends linearly on $\{b_{I_m,\xi}\}_{I_m:\, |I_m|_n=m}$. 
Thus, from \eqref{eq: b_I as pol in lambda}, for each $\ell$-dimensional multi-index $J$, there exists a polynomial $\textcolor{blue}{q_J}$ in $\ell$ variables such that:
\begin{equation}\label{eq: a_I as pol in rho}
   a_{J,\xi} = q_J(\rho_1(\xi),\dots,\rho_\ell(\xi)).
\end{equation} 
For simplicity, we write $q_J(\rho_1(\xi),\dots,\rho_\ell(\xi))=q_J(\rho(\xi))$.
Combining this with the series expression  \eqref{eq: spherical function as a series} for $\varphi_\xi$, we can write
\begin{align}\label{eq: series varpi_xi}
\varphi _{\xi }(x) &=\sum_{J}q_J(\rho(\xi)) \, \varrho (x) ^{J} .   
\end{align}
We recall that the series converges absolutely and uniformly over compacts. Indeed, due to the symmetry $\eqref{eq: sym}$, the series \eqref{eq: series varpi_xi} converges uniformly as a function of $x$ with fixed $\xi$, and also with respect to $\xi$ for fixed $x$.

\begin{definition}\label{def: h_xi}    Let $(K,\mathbb R^n)$ be a Gelfand pair, and let $\varphi _{\xi }$ be an associated bounded spherical function. By using the previous notation, we define the map $\textcolor{blue}{h_\xi}:\mathbb R^\ell\to \mathbb C$
\begin{equation}\label{eq: h_xi}
    h_{\xi }\left( t\right) :=\sum_{J}q_J(\rho(\xi)) \, t^{J}, 
\end{equation}
where $\textcolor{blue}{t}:=\left( t_{1},...,t_{\ell}\right)$ and by abuse of notation $t^{J}:=t_{1}{}^{j_{1}}\cdot ...\cdot t_{\ell}{}^{j_{\ell}}$ for the multi-index $J=(j_1,\dots,j_\ell)$. 
\end{definition}

Notice that for each spherical function $\varphi_\xi$ we have associated $h_\xi$ given by \eqref{eq: h_xi}. Now,
our goal is to prove Proposition \ref{prop: aux sph f}, which can be interpreted as a  version of Theorem \ref{thm: schwarz} for spherical functions. Indeed, for each $\varphi_\xi:\mathbb R^n\to \mathbb C$ we will show that the  function $h_\xi:\mathbb R^\ell\to \mathbb C$ in Definition \ref{def: h_xi} is well-defined, 
satisfies $\varphi_\xi(x)=h_\xi(\rho(x))$ and has certain regularity properties. 
This is the most important result of this section and will be crucial in the next section, but 
we need first the following auxiliary result.

\begin{lemma}\label{lem: pol}\, 
Let  $p_{1},p_{2},\dots,p_\ell$ be non-null homogeneous polynomials in $\mathbb{R}^n$. 
Given a positive
    number $r$, there exists $x_{0}\in \mathbb{R}^{n}$ such that
 $$\max\{|p_{1}(x_0)|, %
 |p_{2}(x_0)|,\dots, |p_\ell(x_0)|\}\geq r.$$    
\end{lemma}

\begin{proof}
It is clear that the results holds for $\ell=1$.
Let $p_1,p_2$ be two homogeneous polynomials. Since $p_{1}(x) p_{2}(x) =0$ only for a
set of zero measure in $\mathbb{R}^{n}$, the product $p_{1}p_{2}$ $\neq 0,$
thus there exists $y_{0}\in\mathbb R^n$, $\left\Vert y_{0}\right\Vert =1$, such that $%
p_{1}( y_{0}) \neq 0$, $p_{2}( y_{0}) \neq 0$. As $p_1,p_2$ are homogeneous polynomials, there exists $m_1,m_2\in \mathbb N$ such that $p_{1}(sy_0) =s^{m_1}p_{1}(y_0)$, $p_{2}(sy_0) =s^{m_2}p_{2}(y_0)$ for all $s\in\mathbb R$. 
Given $r>0$, by choosing $s>0$ sufficiently large, we obtain that for $x_0=sy_0$
$$|p_1(x_0)|\geq r, \quad |p_2(x_0)|\geq r.$$
 This  argument is analogous for any finite set of homogeneous polynomials $p_{j}$'s.
\end{proof}

We proceed to prove Proposition \ref{prop: aux sph f}.

\begin{proof}[Proof of Proposition \ref{prop: aux sph f}]

Let $\varphi _{\xi }$ be a bounded spherical function associated to the Gelfand pair $(K,\mathbb R^n)$ and consider the  $h_\xi$ as in Definition \ref{def: h_xi}.

First, notice that from \eqref{eq: series varpi_xi}, we have the relation $\varphi_\xi(x)=h_\xi(\rho(x))$,  $\forall x\in \mathbb R^n$. 
In the remainder of the proof our goal is to prove not only that $h_\xi$ is a well-defined function, but also that it is very regular. In fact, we will show that it is a real analytic function on $\mathbb R^\ell$, in particular $h_\xi\in C^{\infty}(\mathbb{R}^\ell)$. 

By Lemma \ref{lem: pol}, given a positive number $r$, there exists $x_{0}\in \mathbb{R}%
^{n}$ such that $r\leq \min \left\{ \left\vert \rho _{j}\left( x_{0}\right)
\right\vert \right\} $. Then, for $\| t\| \leq r$,
\begin{align}\label{eq: replace t by rho(x0)2}
&\sum_{J}|q_J(\rho(\xi))|\left\vert t_{1}\right\vert
{}^{j_{1}}\cdot...\cdot\left\vert t_{\ell}\right\vert {}^{j_{\ell}} \leq \sum_{J}|q_J(\rho(\xi))|\left\vert \rho _{1}\left( x_{0}\right) \right\vert
^{j_{1}}\cdot...\cdot\left\vert \rho _{\ell}\left( x_{0}\right) \right\vert ^{j_{\ell}} .   
\end{align}
Since the power series \eqref{eq: series varpi_xi} of $\varphi_\xi$ converges absolutely, then the right hand  side of \eqref{eq: replace t by rho(x0)2} converges. Therefore, the series \eqref{eq: h_xi} converges absolutely and
uniformly for $t$ in  the ball of $ \mathbb{R}^\ell$ of radius $r$ centered at the origin. Since $r>0$ is arbitrary, the series converges uniformly over compact sets of $\mathbb R^\ell$. Thus, $h_\xi$ is real analytic on $\mathbb R^\ell$.
\end{proof}

Finally, as a consequence of the symmetry property \eqref{eq: sym} of the spherical functions, the associated functions defined through the expression \eqref{eq: h_xi} satisfy the following property.

\begin{lemma}\label{lem: h as funct of xi}
    For each fixed $t\in \mathbb R^\ell$, the function $h_\xi(t)$ defined by \eqref{eq: h_xi} is a real analytic as a function on the variable $\xi$.
\end{lemma}
\begin{proof}
    The proof follows similar arguments as in Proposition \ref{prop: aux sph f}, swapping the roles of $t$ and $\xi$ and using the symmetry \eqref{eq: sym} of the spherical functions. 
    Fixed $t\in \mathbb R^\ell$, there exists  $x_{0}\in \mathbb{R}^{n}$ such that 
\begin{align}\label{eq: replace t by rho(x0)}
&\sum_{J}|q_J(\rho(\xi))|\left\vert t_{1}\right\vert
{}^{j_{1}}\cdot...\cdot\left\vert t_{\ell}\right\vert {}^{j_{\ell}} \leq \sum_{J}|q_J(\rho(\xi))|\left\vert \rho _{1}\left( x_{0}\right) \right\vert
^{j_{1}}\cdot...\cdot\left\vert \rho _{\ell}\left( x_{0}\right) \right\vert ^{j_{\ell}} .   
\end{align}
Using the symmetry \eqref{eq: sym}, we can say that the power series \eqref{eq: series varpi_xi} of $\varphi_{x_0}$ as a function of $\xi$, converges absolutely and uniformly over compacts. Therefore, the right hand  side of \eqref{eq: replace t by rho(x0)} converges, because it is bounded above by such absolutely convergent series, and also uniformly over compacts with respect of the variable $\xi$. Therefore, due to the fact that the series \eqref{eq: h_xi} defining $h_\xi(t)$ converges absolutely and
uniformly over compacts with respect of the variable $\xi$ (for fixed $t$), the function  $\xi\mapsto h_\xi(t)$ is a real analytic.

\end{proof}

Finally, notice that following  relations hold true:
$$h_\xi(\rho(x))=\varphi_\xi(x)=\varphi_x(\xi)=h_x({\rho(\xi)}).$$

\subsection{Main Contributions: The Proofs of Theorem \ref{thm: schwarz2} and Corollary \ref{eq: nuestro thm}}\label{sec: main}

In this section, we prove a version of Gerald Schwarz's theorem, stated as Theorem \ref{thm: schwarz}, using only elements from Gelfand theory in Euclidean spaces, which, in our view, is more elementary than the original proof of the theorem.

\begin{definition}\label{def: h}
      Let $(K,\mathbb R^n)$ be a Gelfand pair. For a given $f\in L^1(\mathbb R^n )$, $K$-invariant with $\widehat{f}$ of compact support, we define 
          \begin{equation}\label{eq: h}
        \textcolor{blue}{h(t)} :=\int_{\mathbb R^n} \widehat{f}(\xi)\,  h_{\xi }(t) \, d\xi \qquad \text{for } t=(t_1,\dots,t_\ell)\in \mathbb R^\ell,
        \end{equation}
        where $h_\xi$ is given by \eqref{eq: h_xi}. 
\end{definition}

\begin{proof}[Proof of Theorem \ref{thm: schwarz2}]

Let us consider $f\in L^1(\mathbb R^n )$, $K$-invariant with $\widehat{f}$ of compact support.  
In particular, it holds that $f\in C^\infty(\mathbb R^n)$. 
For such function $f$, consider $h$ as in \eqref{eq: h}.

Let us show that the function $h$ is well-defined.
For each fixed $t$, from Lemma \ref{lem: h as funct of xi}, $\xi\mapsto h_{\xi}(t)$ is continuous, and  thus for $\xi$ in a compact set we have that $h_{\xi}(t)$ is bounded by some constant $C(t)$. Therefore,
\begin{align*}
  \int_{\mathbb R^n} \left\vert \widehat{f}(\xi) 
\, h_{\xi }\left( t\right) \right\vert d\xi&= \int_{sup
p\left( \widehat{f} \, \right) }\left\vert \widehat{f}(\xi) \, 
h_{\xi }\left( t\right) \right\vert d\xi\leq 
\left( \int_{sup
p\left( \widehat{f} \, \right)  } \left\vert 
\widehat{f}(\xi)  \right\vert d\xi \right) C(t)<\infty.
\end{align*}
As a result, $h(t)$ is  well-defined for every $t\in \mathbb R^\ell$.

Now, we will show that $h$ is real analytic. 
In order to see that, we will first check that it is a continuous
function in $\mathbb R^\ell$.
Consider an arbitrary convergent sequence $(t^k)_{k\in \mathbb N}$ in $\mathbb R^\ell$ such that $t^k\to t^*$ for some $t^*\in \mathbb R^\ell$ as $k\to \infty$.
 By Proposition \ref{prop: aux sph f}, for each $\xi$, we know that the function $t\mapsto h_{\xi}(t)$ is continuous in $\mathbb R^\ell$.
Then, we can apply Lebesgue Dominated Convergence Theorem:
\begin{align*}
    \lim_{k\to\infty}h(t^k)&=\lim_{k\to\infty}\int_{sup
p\left( \widehat{f} \, \right) } \widehat{f}(\xi) \, h_\xi(t^k) \, d\xi\\
    &=\int_{sup
p\left( \widehat{f} \, \right) }\widehat{f}(\xi) \, \lim_{k\to\infty}h_\xi(t^k) \, d\xi\\
    &=\int_{sup
p\left( \widehat{f} \, \right) }\widehat{f}(\xi) \, h_\xi(t^*) \, d\xi\\
    &=h(t^*).
\end{align*}

Finally, we recall that a function in $\mathbb C^\ell$ is analytic if it is analytic on each variable.
Let $\gamma$ be closed piecewise-smooth curve 
in $\mathbb C$, where we understand $\mathbb C$ as the first component of $\mathbb C^\ell=\mathbb C\times \dots \times \mathbb C$. Then,  
since $h_\xi(z)$ is analytic, we have
\begin{equation*}
\oint_\gamma h_\xi(z_1,\dots,z_\ell) \, d\gamma(z_1)=0.
\end{equation*}
Then,
\begin{align*}
\oint_\gamma h(z) \, d\gamma(z_1) &= \oint_\gamma \int_{\mathbb R^n} \widehat{f}(\xi) \,  h_{\xi }(z) \, d\xi \,  d\gamma(z_1)\\
&= \int_{sup
p\left( \widehat{f} \, \right) } \widehat{f}(\xi) \underbrace{\oint_\gamma h_{\xi }\left(
z\right) \, d\gamma(z_1) }_{=0}d\xi =0 .
\end{align*}
We were able to interchange the order of integration because we are integrating over compact sets, and so theorems of Fubini-Tonelli hold.
Therefore, by Morera's Theorem, $h$ is analytic on $z_1$. Since the same argument is valid for all variables, it holds that $h$ is analytic as a function in $\mathbb C^\ell$.
Thus, $h$ is not only in $ C^\infty(\mathbb{R}^\ell)$, but also real analytic. 

Finally, under the hypothesis of this theorem we can apply the expression 
\eqref{eq: inv formula with xi} for the inversion formula and the equality \eqref{eq: geltansd and fourier} to obtain 
$f(x)=h(\rho(x))$. Indeed, for every $x\in \mathbb R^n$ yields 
\begin{align*}
    f(x)
    &=\int_{\mathbb{R}^n}\mathcal{F}(f)(\varphi_\xi) \, \varphi_\xi(x) \, d\xi\\
    &= \int_{\mathbb R^n} \mathcal F({f})( \varphi_\xi ) \, h_{\xi }\left(\rho(x)\right) d\xi\\
    &= \int_{\mathbb R^n} \widehat{f}(\xi) \, h_{\xi }\left(\rho(x)\right) d\xi\\
    &=h(\rho(x)), 
\end{align*}
 
where the second equality holds from Proposition \ref{prop: aux sph f}.
\end{proof}

\begin{proof}[Proof of Corollary \ref{eq: nuestro thm}]
    Consider $f=\widehat{g}$ and apply Theorem \ref{thm: schwarz2}. 
    Thus, if we consider
    \begin{align*}
        h(t)=\int_{\mathbb R^n} {\widehat{f}}(\xi) \, h_\xi(t) \, d\xi
        =\int_{\mathbb R^n} \widehat{\widehat{g}}(\xi) \, h_\xi(t) \, d\xi=\int_{\mathbb{R}^n}g(-\xi) \, h_\xi(x) \, d\xi,
    \end{align*}
    it satisfies 
    \begin{equation*}
        \mathcal F({g})(\varphi_\xi)=\widehat{g}(\xi)=f(\xi)=h(\rho(\xi)).
    \end{equation*}
    
\end{proof}

\section{Examples}\label{sec: ex}

The goal of this section is to provide some more explicit flavor on the computation of the coefficients $a_{J,\xi}$ in \eqref{eq: pol}, expressed in terms of a (non-explicit) polynomial $q_J$ in the variables $\rho_1(\xi),\dots, \rho_\ell(\xi)$ as in \eqref{eq: a_I as pol in rho}. 

It is easy to check that 
\begin{equation}\label{eq: dist}
\text{if } \quad  |J|_{\ell,\rho}\not =|J'|_{\ell,\rho}, \quad \text{ then} \quad  (D^{J}\varrho^{J'})(0)=0.     
\end{equation}
However, it is difficult to compute $(D^{J}\varrho^{J'})(0)$ for arbitrary $\ell$ dimensional multi-indexes $J,J'$.

\begin{remark}\label{rem: special case}
Let us denote
    $\textcolor{blue}{\lambda(\xi)}:=\Pi_{k=1}^\ell \lambda_k(\xi)$,
and so we write $\lambda(\xi)^{J}:=\Pi_{k=1}^\ell \lambda_k(\xi)^{j_k}$ for $J=(j_1,\dots, j_\ell)$.
Under the assumption 
\begin{equation}\label{eq: special assuption}
 (D^{J}\varrho^{J'})(0)=
     0,  \qquad \text{ if } \quad J\not\equiv J',
\end{equation}
then we would have the following expressions for the coefficients in \eqref{eq: spherical function as a series}: 
\begin{equation}\label{eq: a}
    a_{J,\xi}=\frac{(D^{J}\varphi_\xi)(0)}{(D^{J}\varrho^{J})(0)}=\frac{\lambda(\xi)^{J}}{(D^{J}\varrho^{J})(0)}\underbrace{\varphi_\xi(0)}_{=1}=\frac{\lambda(\xi)^{J}}{(D^{J}\varrho^{J})(0)}.
\end{equation}
As a result, combining this with \eqref{eq: eigval} in \eqref{eq: spherical function as a series}, we would obtain 
\begin{align}\label{eq: series varpi_xi 2}
\varphi _{\xi }(x) &=\sum_{m=0}^\infty\sum_{{J_m: \, |J_m|_{\ell,\rho}=m}}a_{J_m,\xi} \, \varrho (x) ^{J_m}\notag \\
&=\sum_{m=0}^\infty\sum_{{J_m: \, |J_m|_{\ell,\rho}=m}} \frac{\lambda(\xi)^{J_m}}{(D^{J_m}\varrho^{J_m})(0)} \, \varrho (x)^{J_m}\notag\\
&=\sum_{m=0}^\infty\sum_{{J_m: \, |J_m|_{\ell,\rho}=m}}\frac{i^{m}\varrho(\xi)^{J_m}
}{(D^{J_m}\varrho^{J_m})(0)} \, \varrho(x)^{J_m} 
\end{align}
This expression is very explicit and only requires the computation of the coefficients $(D^J\varrho)(0)$. However, the assumption \eqref{eq: special assuption} does not hold in general.
\end{remark}

When the algebra $\mathcal{P}$ has a single generator,  the relation \eqref{eq: special assuption} holds and we have the nice expression of the spherical functions as in \eqref{eq: series varpi_xi 2}. We will see this in the first examples.

\begin{example}
    Consider the classical theory on Fourier Analysis on the real line $\mathbb{R}$. 
    Here, $K$ is trivial, and we can consider $\rho(x):=x$ as a single generator of the algebra of polynomials $\mathcal{P}(\mathbb{R})$, and the classical derivative operator in one-dimension, $\partial_x$, as a generator of $\mathbb D(\mathbb R)$.
    The bounded spherical functions are the complex exponential functions 
$\varphi_\xi(x)=e^{i\xi x}$ parametrized by $\xi\in \mathbb R$. By using their Taylor expansions, we obtain an expression as in \eqref{eq: series varpi_xi 2}:
\begin{align*}
\varphi_\xi(x)&=e^{i\xi x}=\sum_{k=0}^\infty(\partial_{x}^{k}\varphi_\xi)(0) \, \frac{ x^k}
{k!}=\sum_{k=0}^\infty \; \underbrace{\frac{(i\xi)^k}{k!}}_{a_{k,\xi}}  \, \underbrace{x^k}_{\rho(x)^k}= \sum_{k=0}^\infty \; {\frac{i^k \rho(\xi)^k}{(\partial_x^k\rho(x)^k)(0)}}  \, \rho(x)^k.
\end{align*}
In terms of the Schwarz Theorem, this is a trivial example, as $\varphi_\xi=h_\xi$, and $f=h$ in Theorems \ref{thm: schwarz} or \ref{thm: schwarz2}. 
\end{example}

\begin{example} Let us analyze the Gelfand pair
 $(\mathrm{SO}(n),\mathbb R^n)$, where $K=\mathrm{SO}(n)$ is the special orthogonal 
group. 

On the one hand, 
the algebra of rotational invariant polynomials, $\mathcal{P}(\mathbb R^n)^{\mathrm{SO}(n)}$, is generated by a single polynomial, which can be chosen as  $\rho(x):=\|x\|^2$. Similarly, $\mathbb D(\mathbb R^n)^{\mathrm{SO}(n)}$ is generated by the Laplacian operator $\Delta$.
Notice that 
$$\Delta \|x\|^{2m}=2m(2(m-1)+n)\|x\|^{2(m-1)}, \qquad \qquad \forall m\in \mathbb N$$
and, more generally,
\begin{align*}
  \Delta^{k} \|x\|^{2m}&=2m(2(m-1)+n)\cdot...\cdot(2(m-k+1))(2(m-k)+n)\|x\|^{2(m-k)}, \qquad \forall m,k\in \mathbb N.
\end{align*}
In particular, 
\begin{equation}\label{eq: lo que queremos}
    \left(\Delta^{k} \|x\|^{2m}\right)(0)=0, \qquad \forall k\not=m,
\end{equation}
and 
if $k=m$,
\begin{equation}\label{eq: explicit comp}
\Delta^{k} \|x\|^{2k}= (2k)!!n(n+2)\cdot...\cdot(n+2(k-1)),    
\end{equation}
where $\textcolor{blue}{2k!!}:=2k\cdot 2(k-2)\cdot2(k-4)\cdot...\cdot 4\cdot 2 $. Thus, in this case, we are under the hypothesis \eqref{eq: special assuption}.

On the other hand,  the algebra $L^1(\mathbb{R}^n)^{\mathrm{SO}(n)}$ is that of all rotational invariant functions in $L^1(\mathbb{R}^n)$, that is, the algebra of integrable radial functions. Its spectrum $\Lambda$ can be identified with the half-line $\{\|\xi\|e_1: \, s\in \mathbb R_{\geq 0}\}\subset \mathbb{R}^n$, for $e_1=(1,0,...,0)\in\mathbb R^n$, 
as we have that for every $\xi\in \mathbb{R}^n$ the bounded spherical functions $\varphi_\xi$ can be obtained as 
$$\varphi_{\xi}(x)=\int_{\mathrm{SO}(n)}e^{i\langle x,k\cdot \xi\rangle} \, dk=\int_{\mathbb{S}^{n-1}}e^{i\|\xi\|\langle x,\theta\rangle} \, d\theta,$$ where $d\theta$ is the uniform probability measure on the sphere $\mathbb{S}^{n-1}$. We have that, $\varphi_{\xi_1}=\varphi_{\xi_2}$ as long as $\|\xi_1\|=\|\xi_2\|$. 
It holds that 
\begin{equation}\label{eq: explicit eigvenval}
    ({\Delta^{k}\varphi_\xi)(0)}={(i\|\xi\|)^{2k}}=i^{2k}\rho(\xi)^k.
\end{equation}
(Notice that $|k|_{1,\rho}=2k$.)
Moreover, it is well-known the following series expression for each bounded spherical function (see, for e.g., \cite{Dijk}):
\begin{equation}\label{eq: sph son}
    \varphi_\xi(x)=\sum_{k}\underbrace{\frac{\Gamma(2/n)}{k!\Gamma(k+n/2)}\frac{(i\|\xi\|)^{2k}}{2^{2k}}}_{a_{k,\xi}} \, \underbrace{\|x\|^{2k}}_{\rho(x)^k}.
\end{equation}
Thus, we already have an explicit expression for the coefficients $a_{k,\xi}$. 
The goal now is to show that the formula \eqref{eq: sph son} from the literature coincides with an expression like \eqref{eq: series varpi_xi 2}.
Indeed, since we are under hypothesis \eqref{eq: special assuption}, then using \eqref{eq: a} we have
$$a_{k,\xi}=\frac{(\Delta^{k}\varphi_\xi)(0)}{(\Delta^{k}\rho^{k})(0)}=\frac{i^{2k}\|\xi\|^k}{(2k)!!n(n+2)\cdot...\cdot(n+2(k-1))},$$
where we have used \eqref{eq: explicit comp} and \eqref{eq: explicit eigvenval}. 
Therefore, to match with the formula \eqref{eq: sph son},  one needs to verify the identity
\begin{equation}\label{eq: aux1}
    \frac{\Gamma(n/2)}{\Gamma(k+n/2)k!2^{2k}}=\frac{1}{(2k)!!n(n+2)\cdot...\cdot(n+2(k-1))}.
\end{equation}
It indeed holds true because, on the one hand, it is easy to check that $$k!2^k=(2k)!!$$ and,  on the other hand, by applying iteratively the identity  $\Gamma(z+1)=z\Gamma(z)$ for the Gamma function, one can reach
\begin{equation*}
    \Gamma(k+n/2)=\frac{n}{2}(\frac{n}{2}+1)\cdot...\cdot(\frac{n}{2}+(k-1))\Gamma(n/2)=\frac{n(n+2)\cdot ...  \cdot n+2(k-1)}{2^k}\Gamma(n/2).
\end{equation*}

Finally, in this particular example we have the explicit formula
\begin{align*}
    h_\xi(t)=\sum_{k}{\frac{\Gamma(2/n)}{k!\Gamma(k+n/2)}\frac{(i\|\xi\|)^{2k}}{2^{2k}}} \, t^k =\sum_{k}\frac{i^{2k}\|\xi\|^k}{(2k)!!n(n+2)\cdot...\cdot(n+2(k-1))} \, t^k.
\end{align*}

\end{example}

\begin{example}
    Consider the Gelfand pair $(\mathbb Z_2,\mathbb R^2)$. Then, $L^1(\mathbb R^2)^{\mathbb Z_2}$ is the space of integrable even functions. If $\mathcal{P}(\mathbb R^2)=span\{x_1,x_2\}$, then $\mathcal{P}(\mathbb R^2)^{\mathbb Z_2}=span\{x_1^2,x_1x_2,x_2^2\}$. 
    Let us denote, in this case, 
    $$\rho_1(x_1,x_2):=x_1^2 \qquad \rho_2(x_1,x_2):=x_1x_2 \qquad \rho_3(x_1,x_2):=x_2^2.$$
    We view this as an interesting example, since, on the one hand,  $n=2<3=\ell$, that is, roughly speaking, the number of generators of $\mathcal{P}(\mathbb R^2)^{\mathbb Z_2}$ is greater than the number of `variables'. On the other hand, in this example the generators of $\mathcal{P}(\mathbb R^2)^{\mathbb Z_2}$ are not algebraically independent, in fact, they satisfy the relation
    \begin{equation}\label{eq: relation z2}
    (\rho_2)^2=\rho_1\rho_3    .
    \end{equation}
    Let $\varphi_\xi$ be a bounded spherical function of the pair $(\mathbb Z_2, \mathbb R^2)$ written as a power series as in \eqref{eq: spherical function as a series}
    \begin{equation}\label{eq: example eq 1}
        \varphi_\xi(x)=\sum_{J=(j_1,j_2,j_3)}a_{J,\xi} \, \rho_1(x)^{j_1}\rho_2(x)^{j_2}\rho_3(x)^{j_3} \qquad \forall x=(x_1,x_2).
    \end{equation}
    In order to get a more explicit expression for the coefficients of the series above, we proceed as follows.
    \begin{enumerate}
        \item First, let us make the choice that we replace every even power of $\rho_2$, i.e., every $\rho_2(x)^{2k}$ for $k\in \mathbb N$, by the same factor but written in terms of $\rho_1$ and $\rho_3$, i.e., $\rho_1(x)^k\rho_3(x)^k$. After this, we get a series expression for \eqref{eq: example eq 1} in terms of the generators $\rho_1,\rho_2,\rho_3$, but `clean' of relations among them. Precisely, given a multi-index 
    $J=(j_1,j_2,j_3)$, we write $j_2=2k_J+\omega_J$ for unique $\omega_J\in \{0,1\}$ and $k_J\in \mathbb N$, and consider a second multi-index $I:=(i_1,i_2,i_3)$ such that
    \begin{equation}
        i_1:=k_J+j_1, \qquad  i_2:=\omega_J, \qquad i_3:=k_J+j_3 .
    \end{equation} 
    Then, $\rho_1(x)^{j_1}\rho_2(x)^{j_2}\rho_3(x)^{j_3}=\rho_1(x)^{i_1}\rho_2(x)^{i_1}\rho_1(x)^{i_1}$ but we only write the right-hand side expression.

    \item Secondly, by taking common factors, the terms  in \eqref{eq: example eq 1} associated to the monomials $\rho(x)^J$ and $\rho(x)^I$ are combined in only one term
    $\left(a_{J,\xi}+a_{I,\xi}\right)\rho(x)^{I}$. That is, we can rewrite \eqref{eq: example eq 1} as
    \begin{equation}\label{eq: example eq 2}
        \varphi_\xi(x)=\sum_{\substack{
   I=(u+k,\omega,v+k)\\ \text{such that}\\
   \omega \in \{0,1\}, \,  
   u,v,k\in \mathbb N}}\left(a_{(u,2k+\omega,v),\xi}+a_{(u+k,\omega,v+k),\xi}\right)\, \underbrace{\rho_1(x)^{u+k}\rho_2(x)^{\omega}\rho_3(x)^{u+k}}_{\rho(x)^I} \quad \forall x=(x_1,x_2),
    \end{equation}
    where $\xi=(\xi,\xi_2)\in \mathbb R^2$.
    \item Now, consider two arbitrary multi-indexes $I=(i_1,i_2,i_3)$, $I'= (i_1',i_2',i_3')$ that appear in \eqref{eq: example eq 2}, that is, with $i_2,i_2'\in \{0,1\}$. We will show that 
    \begin{equation}
        (D^I\rho^{I'})(0)\not=0 \qquad \text{if and only if} \qquad I\equiv I'.
    \end{equation}
    Indeed, the expression
    \begin{align}\label{eq: is not 0}
        (D^I\rho^{I'})(0)=
    \left(\partial_{x_1}^{i_1+i_2} \, \partial_{x_2}^{i_3+i_2} \, (x_1)^{i_1'+i_2'}(x_2)^{i_3'+i_2'}\right)(0)
    \end{align}
    is not null if and only if 
    \begin{equation}\label{eq: parity}
        i_1+i_2=i_1'+i_2' \qquad \text{and} \qquad i_3+i_2=i_3'+i_2'.
    \end{equation}
    
     It is trivial to see that if $I\equiv I'$, then $(D^I\rho^{I'})(0)\not=0$. 
     Thus, only the other direction is left to prove. Arguing by contrapositive, let us assume $I\not\equiv I'$, and let us show that in such case $(D^I\rho^{I'})(0)=0$.       
    Since  $\mathrm{deg}(\rho_k)=2$ for every $k=1,2,3$, from \eqref{eq: dist} we have that if $i_1+i_2+i_3\not=i_1'+i_2'+i_3'$, then $(D^I\rho^{I'})(0)=0$. As a result, we can suppose that $i_1+i_2+i_3=i_1'+i_2'+i_3'$. Let us separate in two cases.    
    First, if $i_2=i_2'$, since $I\not\equiv I'$, then $i_1\not=i_1'$ or $i_3\not=i_3'$, and in either case \eqref{eq: is not 0} is not null. Second, if $i_2\not=i_2'$, since their only two possible values are $0$ and $1$, we have that \eqref{eq: parity} holds if and only if $i_1$ and $i_1'$ have different parity, as well as $i_3$ and $i_3'$ (i.e., one is even and the other is odd). However, if so it contradicts the equality $i_1+i_2+i_3=i_1'+i_2'+i_3'$ as we fall under one of the following three cases:
\begin{table}[h]
\centering
\renewcommand{\arraystretch}{1} 
\begin{tabular}{c|c|c||c|c|c}
\, $i_1$ \, &  $i_2$ \,  &  $i_3$  & \, $i_1'$ \,  &  \, $i_2'$\,  & $i_3'$  \\
\hline
\multicolumn{3}{c}{$\underbrace{\text{odd} + 0 + \text{odd}}_{\text{even}}$} &
\multicolumn{3}{c}{$\underbrace{\text{even} + 1 + \text{even}}_{\text{odd}}$} \\
\multicolumn{3}{c}{$\underbrace{\text{even} + 0 + \text{even}}_{\text{even}}$} &
\multicolumn{3}{c}{$\underbrace{\text{odd} + 1 + \text{odd}}_{\text{odd}}$} \\
\multicolumn{3}{c}{$\underbrace{\text{odd} + 0 + \text{even}}_{\text{odd}}$} &
\multicolumn{3}{c}{$\underbrace{\text{even} + 1 + \text{odd}}_{\text{even}}$} \\
\end{tabular}
\end{table}

Therefore, \eqref{eq: parity} is not satisfied and so $(D^I\rho^{I'})(0)=0$.

\item Finally, from \eqref{eq: example eq 2} we have that for every multi-index $I=(u+k,\omega,v+k)$, with 
   $\omega \in \{0,1\}$, $u,v,k\in \mathbb N$
\begin{align*}
    \lambda(\xi)^I=D^I\varphi_\xi(0)=\left(a_{(u,2k+\omega,v),\xi}+a_{(u+k,\omega,v+k),\xi}\right)(D^I\rho^I)(0).
\end{align*}
Thus, using \eqref{eq: eigval},
$$\left(a_{(u,2k+\omega,v),\xi}+a_{(u+k,\omega,v+k),\xi}\right)=\frac{i^{2(u+v+2k+\omega)}\rho^I(\xi)}{(D^I\rho^I)(0)}=\frac{-\rho_1(\xi)^{u+k}\rho_2(\xi)^{\omega}\rho_3(\xi)^{v+k}}{(u+k+\omega)!(v+k+\omega)!}.$$
    \end{enumerate}
    
  Hence,
    \begin{equation*}\label{eq: example eq 3}
        \varphi_\xi(x)=\sum_{\substack{
   I=(u+k,\omega,v+k)\\ \text{such that}\\
   \omega \in \{0,1\}, \,  
   u,v,k\in \mathbb N}}\frac{-\overbrace{\rho_1(\xi)^{u+k}\rho_2(\xi)^{\omega}\rho_3(\xi)^{v+k}}^{\rho(\xi)^I}}{(u+k+\omega)!(v+k+\omega)!}\, \rho_1(x)^{u+k}\rho_2(x)^{\omega}\rho_3(x)^{u+k},
    \end{equation*}
  
  and 
        \begin{equation*}\label{eq: example eq 4}
        h_\xi(t_1,t_2,t_3)=\sum_{\substack{
   I=(u+k,\omega,v+k)\\ \text{such that}\\
   \omega \in \{0,1\}, \,  
   u,v,k\in \mathbb N}}\frac{-\rho(\xi)^I}{(u+k+\omega)!(v+k+\omega)!} \, t_1^{u+k} \, t_2^\omega  \, t_3^{v+k}.
    \end{equation*}

\end{example}

\subsection*{Acknowledgment} The authors thank Leandro Cagliero for insightful conversations.

\bibliographystyle{plain}

\noindent\textbf{\normalsize{Rocío Díaz Martín}}\\
\small{Department of Mathematics, Florida State University, Tallahassee, FL 32306, USA}\\
\texttt{rdiazmartin@fsu.edu}

\medskip

\noindent\textbf{\normalsize{Linda Saal}}\\
\small{CIEM -- CONICET, X5000 Córdoba, Argentina}\\
\texttt{linda.saal@gmail.com}

\end{document}